\documentclass[11pt]{amsart}

\usepackage{amsmath}
\usepackage{amsthm}
\usepackage{amssymb}
\usepackage{upref}
\usepackage{amscd}
\usepackage[all]{xy}

\def\Z{\mathbb Z}
\def\F{\mathbb F}

\def\C{\mathbb C}

\def\Qa{\overline{\mathbb Q}\mspace{1mu}}
\def\Qp{{\mathbb Q}_p}
\def\Zp{{\mathbb Z}_p}
\def\Fp{{\mathbb F}_{\!p}}
\def\Qpa{\overline{\mathbb Q}_p}

\def\eqdef{\underset{\text{\tiny def}}{=}}

\DeclareMathOperator{\Ker}{Ker}

\DeclareMathOperator{\Endom}{End}
\DeclareMathOperator{\Hom}{Hom}
\DeclareMathOperator{\aut}{Aut}

\DeclareMathOperator{\Isom}{Isom}
\DeclareMathOperator{\Mat}{Mat}

\DeclareMathOperator{\pchar}{P_{char}}

\DeclareMathOperator{\gal}{Gal}

\DeclareMathOperator{\D}{{\bf D}}
\DeclareMathOperator{\Dcr}{{\bf D}^*_{\text{\tiny cris}}}
\DeclareMathOperator{\Vcr}{{\bf V}^*_{\text{\tiny cris}}}
\DeclareMathOperator{\Fil}{Fil}

\DeclareMathOperator{\adm}{ad}

\begingroup
\theoremstyle{plain}
\newtheorem{thm}{Theorem}[section]
\newtheorem{prop}[thm]{Proposition}

\theoremstyle{definition}

\newtheorem{rem}[thm]{Remark}

\endgroup

\begin{document}

\title[Supersingular abelian surfaces with non semisimple Tate module]{Abelian surfaces with supersingular good reduction and non semisimple Tate module}

\author{Maja Volkov}

\begin{abstract}
We show the existence of abelian surfaces $\mathcal{A}$ over $\Qp$ having good reduction with supersingular special fibre whose associated $p$-adic Galois module $V_p(\mathcal{A})$ is not semisimple.
\end{abstract}

\address{Universit\'e de Mons, D\'epartement de Math\'ematique, Place du Parc 20, B-7000 Mons, Belgium.}

\email{maja.volkov@umons.ac.be}

\maketitle



\tableofcontents

\section*{Introduction}
Fix a prime number $p$ and an algebraic closure $\Qpa$ of $\Qp$. Write $G=\gal(\Qpa /\Qp)$ for the absolute Galois group of $\Qp$. For a $d$-dimensional abelian variety ${\mathcal A}$ over $\Qp$ let ${\mathcal A}[p^n]$ be the group of $p^n$-torsion points with values in $\Qpa$ and 
$$V_p({\mathcal A})\eqdef\,\Qp\otimes_{\Z_p}\varprojlim_{n\geq 1}{\mathcal A}[p^n].$$
This is a $2d$-dimensional $\Qp$-vector space on which $G$ acts linearly and continuously. We want to consider the following problem: find abelian varieties ${\mathcal A}$ over $\Qp$ having good reduction with supersingular special fibre and such that the Galois module $V_p({\mathcal A})$ is {\em not} semisimple. In this paper we show the existence of two such varieties with nonisogenous special fibres for the least dimension possible, namely for $d=2$. In fact our procedure easily generalises to any $d\geq2$, however we stick to surfaces as they furnish low-dimensional hence simple to describe representations.  

\medskip 
The existence of such surfaces follows from the characterisation of $p$-adic representations of $G$ arising from abelian varieties with (tame) potential good reduction obtained in~\cite{Vo}, and indeed provides an example of application of this result. In order to explicitely describe our objects we use Fontaine's contravariant functor establishing an equivalence between crystalline $p$-adic representations of $G$ and admissible filtered $\varphi$-modules. In section~\ref{sec:method} we briefly review this theory as well as the characterisation in~\cite{Vo} (Theorem~\ref{abvarthm}), and outline the general strategy. In sections~\ref{sec:prodellip} and~\ref{sec:absurf} we construct two filtered $\varphi$-modules arising from abelian surfaces over $\Qp$ with good reduction that enjoy the required properties (Propositions~\ref{prodellipprop} and~\ref{absurfprop}).

\section{The general method}
\label{sec:method}
Recall from~\cite{Fo2} that the objects $D$ in the category ${\bf MF}_{\Qp}(\varphi)$ of filtered $\varphi$-modules are finite dimensional $\Qp$-vector spaces together with a Frobenius map $\varphi\in\aut_{\Qp}(D)$ and a decreasing filtration $\Fil=(\Fil^iD)_{i\in \Z}$ on $D$ by subspaces such that $\Fil^iD=D$ for $i\ll 0$ and $\Fil^iD=0$ for $i\gg 0$, and the morphisms are $\Qp$-linear maps commuting with $\varphi$ and preserving the filtration. The dual of $(D,\Fil)$ is the $\Qp$-linear dual $D^*$ with $\varphi_{D^*}=\varphi^{*\,-1}$ and $\Fil^iD^*$ consists of linear forms on $D$ vanishing on $\Fil^jD$ for all $j>-i$. The Tate twist $D\{-1\}$ of $(D,\Fil)$ is $D$ as a $\Qp$-vector space with $\varphi_{D\{-1\}}=p\varphi$ and $\Fil^iD\{-1\}=\Fil^{i-1}D$. The filtration $\Fil$ has Hodge-Tate type $(0,1)$ if $\Fil^iD=D$ for $i\leq 0$, $\Fil^iD=0$ for $i\geq 2$, and $\Fil^1D$ is a nontrivial subspace. The full subcategory ${\bf MF}_{\Qp}^{\adm}(\varphi)$ of ${\bf MF}_{\Qp}(\varphi)$ consists of objects $(D,\Fil)$ satisfying a property relating the Frobenius with the filtration, called admissibility and defined as follows. For a $\varphi$-stable sub-$\Qp$-vector space $D^{\prime}$ of $D$ consider the Hodge and Newton invariants
$$t_H(D^{\prime})\eqdef\,\sum_{i\in\Z}i\dim_{\Qp}\!\bigl(D^{\prime}\!\cap\Fil^{i}\!D\big{/}D^{\prime}\!\cap\Fil^{i+1}\!D\bigr)\quad\text{ and }\quad t_N(D^{\prime})\eqdef\,v_p(\det\varphi_{\!\mid D^{\prime}})$$
where $v_p$ is the normalised $p$-adic valuation on $\Qp$. Then $(D,\Fil)$ is admissible if 
\begin{itemize}
\item[(i)] $t_H(D)=t_N(D)$
\item[(ii)] $t_H(D^{\prime})\leq t_N(D^{\prime})$ for any sub-$\Qp[\varphi]$-module $D^{\prime}$ of $D$. 
\end{itemize} 
A sub-$\Qp[\varphi]$-module $D^{\prime}$ endowed with the induced filtration $\Fil^{i}D^{\prime}=D^{\prime}\cap\Fil^{i}D$ is a subobject of $(D,\Fil)$ in ${\bf MF}_{\Qp}^{\adm}(\varphi)$ if and only if $t_H(D^{\prime})= t_N(D^{\prime})$. 

Let $B_{\text{\tiny cris}}$ be the ring of $p$-adic periods constructed in~\cite{Fo1} and for a $p$-adic representation $V$ of $G$ put
$$\Dcr(V)\eqdef\,\Hom_{\Qp[G]}(V,B_{\text{\tiny cris}}).$$
We always have $\dim_{\Qp}\Dcr(V)\leq \dim_{\Qp}V$ and $V$ is said to be crystalline when equality holds. The functor $V\mapsto\Dcr(V)$ establishes an anti-equivalence between the category of crystalline $p$-adic representations of $G$ and ${\bf MF}_{\Qp}^{\adm}(\varphi)$, a quasi-inverse being $\Vcr(D,\Fil)=\Hom_{\varphi,\Fil}(D,B_{\text{\tiny cris}})$ (\cite{Co-Fo}). These categories are well-suited to our problem since for an abelian variety ${\mathcal A}$ over $\Qp$ the $G$-module $V_p({\mathcal A})$ is crystalline if and only if ${\mathcal A}$ has good reduction (\cite{Co-Io} Thm.4.7 or~\cite{Br} Cor.5.3.4.).

\medskip
A $p$-Weil number is an algebraic integer such that all its conjugates have absolute value $\sqrt{p}$ in $\C$. Call a monic polynomial in $\Z[X]$ a $p$-Weil polynomial if all its roots in $\Qa$ are $p$-Weil numbers and its valuation at $X^2-p$ is even. Consider the following conditions on a filtered $\varphi$-module $(D,\Fil)$ over $\Qp$:
\begin{itemize}
\item[(1)] $\varphi$ acts semisimply and $\pchar(\varphi)$ is a $p$-Weil polynomial
\item[(2)] the filtration has Hodge-Tate type $(0,1)$
\item[(3)] there exists a nondegenerate skew form on $D$ under which $\varphi$ is a $p$-similitude and $\Fil^1D$ is totally isotropic.
\end{itemize}
Recall that $\varphi$ is a $p$-similitude under a bilinear form $\beta$ if $\beta(\varphi x,\varphi y)=p\beta(x,y)$ for all $x,y\in D$ and $\Fil^1D$ is totally isotropic if $\beta(x,y)=0$ for all $x,y\in\Fil^1D$. The map sending $\delta\in\Isom_{\Qp}(D^*,D)$ to $\beta:(x,y)\mapsto\delta^{-1}(x)(y)$ identifies the antisymmetric isomorphisms of filtered $\varphi$-modules from $D^*\{-1\}$ to $D$ with the forms satisfying (3). A $\Qp$-linear map $\delta:D^*\rightarrow D$ is an antisymmetric morphism in ${\bf MF}_{\Qp}(\varphi)$ if $\delta^*=-\delta$ (under the canonical isomorphism ${D^*}^*\simeq D$), $\varphi\delta=p\delta\varphi^{*\,-1}$, and $\delta(\Fil^1D)^{\perp}\subseteq\Fil^1D$. 

\begin{rem}
Let $\Hom_{\varphi}^{\text{a}}(D^*\{-1\},D)$ be the $\Qp$-vector space of antisymmetric $\varphi$-module morphisms from $D^*\{-1\}$ to $D$ and pick any $\delta\in\Isom_{\varphi}^{\text{a}}(D^*\{-1\},D)$. Then $\alpha^{\dag}=\delta\alpha^*\delta^{-1}$ defines an involution $\dag$ on $\Endom_{\varphi}(D)$ and the map $\alpha\mapsto\alpha\delta$ establishes an isomorphim $\Endom_{\varphi}(D)^{\dag}\xrightarrow{\sim}\Hom_{\varphi}^{\text{a}}(D^*\{-1\},D)$ where $\Endom_{\varphi}(D)^{\dag}$ is the subspace of elements fixed by $\dag$. 
\end{rem}

\smallskip
\begin{thm}[\cite{Vo} Corollary 5.9]
\label{abvarthm}
Let $V$ be a crystalline $p$-adic representation of $G$. The following are equivalent:
\begin{itemize}
\item[(i)] there is an abelian variety $\mathcal{A}$ over $\Qp$ such that $V\simeq V_p(\mathcal{A})$ 
\item[(ii)] $\D^*_{\text{\em\tiny cris}}(V)$ satisfies conditions $(1)$, $(2)$ and $(3)$.
\end{itemize}
\end{thm}

Note that the restriction $p\neq 2$ in~\cite{Vo} Theorem 5.7 and its Corollary 5.9 is unnecessary as Kisin shows that a crystalline representation with Hodge-Tate weights in $\{0,1\}$ arises from a $p$-divisible group unrestrictidly on the prime $p$ (\cite{Ki} Thm.0.3).

Let $\mathcal{A}$ be an abelian variety over $\Qp$ having good reduction and $(D,\Fil)=\Dcr(V_p(\mathcal{A}))$. The $\varphi$-module $D$ satisfies (1) by the Weil conjectures for abelian varieties over $\F_p$. Tate's theorem on endomorphisms of the latter (see~\cite{Wa-Mi} II) shows that the isomorphism class of the $\varphi$-module $D$, given by semisimplicity by $\pchar(\varphi)$, determines the isogeny class of the special fibre of $\mathcal{A}$ over $\Fp$. Any polarisation on ${\mathcal A}$ induces a form on $D$ satisfying (3) and the filtration satisfies (2) by the Hodge decomposition for $p$-divisible groups and (3).

Conversely let $V$ be a crystalline $p$-adic representation of $G$ such that $\Dcr(V)$ satisfies (1), (2), (3). From (1) the Honda-Tate theory (\cite{Ho-Ta}) furnishes an abelian variety $A$ over $\F_p$ with the right Frobenius. From (2) Kisin's result~\cite{Ki} furnishes a $p$-divisible group over $\Zp$ lifting $A(p)$. The Serre-Tate theory of liftings then produces a formal abelian scheme $\mathcal{A}$ over $\Zp$ with special fibre isogenous to $A$. Finally (3) furnishes a polarisation on $\mathcal{A}$ which ensures by Grothendieck's theorem on algebraisation of formal schemes (\cite{Gr} 5.4.5) that $\mathcal{A}$ is a true abelian scheme. The proof of Theorem 5.7 in~\cite{Vo} details this construction.

\medskip
Thus we want to construct an admissible filtered $\varphi$-module $(D,\Fil)$ over $\Qp$ satisfying conditions (1), (2), (3) of theorem~\ref{abvarthm} and such that
\begin{itemize}
\item[(a)] $\pchar(\varphi)$ is a supersingular $p$-Weil polynomial
\item[(b)] $(D,\Fil)$ is not semisimple.
\end{itemize}

Recall that a $p$-Weil polynomial is supersingular if its roots are of the form $\zeta\sqrt{p}$ with $\zeta\in\Qa$ a root of unity, and that an abelian variety $A$ over $\Fp$ is supersingular if and only if the characteristic polynomial of its Frobenius is supersingular. Regarding (a) in section~\ref{sec:prodellip} we take $\pchar(\varphi)(X)=(X^2+p)^2$ which is the characteristic polynomial of the Frobenius of the product of a supersingular elliptic curve $E$ over $\Fp$ with itself. In section~\ref{sec:absurf} we take $\pchar(\varphi)(X)=X^4+pX^2+p^2$ which is the characteristic polynomial of the Frobenius of a simple supersingular abelian surface over $\Fp$.

Regarding (b) we assume $p\equiv 1\bmod 3\Z$ in section~\ref{sec:absurf}. In each (a)-case we find a subobject $D_1$ of $(D,\Fil)$ in ${\bf MF}_{\Qp}^{\adm}(\varphi)$ and a quotient object $D_2$ (endowed with the quotient filtration $\Fil^iD_2=\Fil^iD\bmod D_1$) such that the sequence 
$$\text{\sc (s)}\qquad\xymatrix{1\ar[r]  &  D_1\ar[r]^{\text{\tiny incl}} & D\ar[r]^{\text{\tiny proj}} &  D_2\ar[r] & 1}$$
is exact and $D_2$ is {\em not} a subobject. Thus $\text{\sc (s)}$ does not split and therefore $(D,\Fil)$ is not semisimple. Of course when $(D,\Fil)\simeq\Dcr(V_p(\mathcal{A}))$ this means that there is a nonsplit short exact sequence of $G$-modules
$$\xymatrix{1\ar[r]  &  V_2\ar[r] & V_p(\mathcal{A})\ar[r] &  V_1\ar[r] & 1}$$
with $V_i\simeq\Vcr(D_i)$ for $i=1,2$, and it follows that $V_p(\mathcal{A})$ is not a semisimple $G$-module.

\section{A lift of the twofold product of a supersingular elliptic curve}
\label{sec:prodellip}
Consider the filtered $\varphi$-module $(D,\Fil)$ over $\Qp$ defined as follows. There is a $\Qp$-basis $\mathcal{B}=(x_1,y_1,x_2,y_2)$ for $D$ so that 
$$D=\Qp x_1\oplus\Qp y_1\oplus\Qp x_2\oplus\Qp y_2$$ 
is a 4-dimensional $\Qp$-vector space. The matrix of $\varphi$ over $\mathcal{B}$ is 
$$\Mat_{\mathcal{B}}(\varphi)=\,\begin{pmatrix}
0 & -p & 0 & 0 \\
1 & 0  & 0 & 0 \\
0 & 0  & 0 & -p \\
0 & 0  & 1 & 0
\end{pmatrix}\in\,GL_4(\Qp)$$
and the filtration is given by  
$$\Fil^0D=D,\quad\Fil^1D=\,\Qp x_1\oplus\Qp(y_1+x_2),\quad\Fil^2D=0.$$

\begin{prop}
\label{prodellipprop}
There is an abelian surface $\mathcal{A}$ over $\Qp$ such that $(D,\Fil)\simeq\D^*_{\text{\em\tiny cris}}(V_p(\mathcal{A}))$. Further 
\begin{itemize}
\item[(a)] $\mathcal{A}$ has good reduction with special fibre isogenous to the product of two supersingular elliptic curves over $\Fp$
\item[(b)] the $G$-module $V_p(\mathcal{A})$ is not semisimple. 
\end{itemize}
\end{prop}

\begin{proof}
The filtration has Hodge-Tate type $(0,1)$ with $\dim\Fil^1D=2$ and $\det\varphi=p^2$ hence $t_H(D)=2=t_N(D)$. Since $\pchar(\varphi)(X)=(X^2+p)^2$ the nontrivial $\varphi$-stable subspaces of $D$ are the $D_i=\Qp x_i\oplus\Qp y_i$ for $i=1,2$ both having Newton invariant $t_N(D_i)=1$. However $D_1\cap\Fil^1D=\Qp x_1$ whereas $D_2\cap\Fil^1D=0$, so $t_H(D_1)=1$ and $t_H(D_2)=0$. Therefore $(D,\Fil)$ is admissible, $D_1$ is a subobject, $D_2$ is a quotient that is not a subobject, the short exact sequence $\text{\sc (s)}$ does not split and $(D,\Fil)$ is not semisimple.

The action of $\varphi$ is semisimple and $\pchar(\varphi)=\pchar(\text{Fr}_{E})^2$ where $E$ is a supersingular elliptic curve over $\Fp$ with $\pchar(\text{Fr}_{E})(X)=X^2+p$. Thus $(D,\Fil)$ satisfies condition (1) of theorem~\ref{abvarthm} as well as condition (a) of section~\ref{sec:method} and it obviously satisfies (2). It remains to check condition (3) that is to find a $\delta\in\Isom_{\Qp}(D^*,D)$ satisfying $\delta^*=-\delta$, $\varphi\delta=p\delta\varphi^{* -1}$, and $\delta(\Fil^1D)^{\perp}=\Fil^1D$. Let $\mathcal{B}^*=(x_1^*,y_1^*,x_2^*,y_2^*)$ be the dual basis of $\mathcal{B}$ for $D^*$ where $z^*$ is the linear form on $D$ sending $z\in D$ to $1$ and vanishing on all vectors noncolinear to $z$. The matrix of $p\varphi^{*\,-1}$ over $\mathcal{B}^*$ is  
 $$p\Mat_{\mathcal{B}}(\varphi^{-1})^{t}=\,\begin{pmatrix}
0 & -1 & 0 & 0 \\
p & 0  & 0 & 0 \\
0 & 0  & 0 & -1 \\
0 & 0  & p & 0
\end{pmatrix}$$
where $M^t$ is the transpose of $M$ and 
$$(\Fil^1D)^{\perp}=\,\Qp y_2^*\oplus\Qp(y_1^*-x_2^*).$$
Let $\delta:D^*\rightarrow D$ be the $\Qp$-linear morphism with matrix over the bases $\mathcal{B}^*$ and $\mathcal{B}$ 
$$\Mat_{\mathcal{B}^*\!\!,\mathcal{B}}(\delta)=\,\begin{pmatrix}
0 & 0 & 0 & -1 \\
0 & 0  & 1 & 0 \\
0 & -1  & 0 & 0 \\
1 & 0  & 0 & 0
\end{pmatrix}.$$
Then $\delta$ is invertible and satisfies the relations $\delta^*=-\delta$ and $\varphi\delta=p\delta\varphi^{*\,-1}$. Further $\delta(\Fil^1D)^{\perp}=\delta\bigl(\Qp y_2^*\oplus\Qp(y_1^*-x_2^*)\bigr)=\Qp x_1\oplus\Qp(y_1+x_2)=\Fil^1D$. 
\end{proof}

\begin{rem}
Any 2-dimensional object satisfying conditions (1) and (2) of theorem~\ref{abvarthm} also satisfies condition (3). Hence theorem~\ref{abvarthm} applied to the admissible filtered $\varphi$-modules $(D_1,\Fil^iD\cap D_1)$ and $(D_2,\Fil^iD\bmod D_1)$ shows the existence of elliptic schemes $\mathcal{E}_i$ over $\Zp$ such that $D_i\simeq\Dcr(V_p(\mathcal{E}_i))$ for $i=1,2$. The special fibres of the $\mathcal{E}_i$ are $\Fp$-isogenous to $E$. Thus we obtain a nonsplit exact sequence of $G$-modules 
$$\xymatrix{1\ar[r]  &  V_p(\mathcal{E}_2)\ar[r] & V_p(\mathcal{A})\ar[r] &  V_p(\mathcal{E}_1)\ar[r] & 1}.$$
By Tate's full faithfulness theorem~\cite{Ta} the $G$-module $V_p(\mathcal{A})$ determines the $p$-divisible group $\mathcal{A}(p)$ over $\Zp$ up to isogeny, therefore $\mathcal{A}(p)$ is not $\Zp$-isogenous to $\mathcal{E}_1(p)\times\mathcal{E}_2(p)$.  
\end{rem}

\medskip
\begin{rem}
The same construction works starting with the square of any supersingular $p$-Weil polynomial of degree two (when $p\geq 5$ there is only $X^2+p$ but when $p=2$ or $3$ there are also the $X^2\pm pX+p$). However it fails when dealing with the product of two distinct such. Indeed let $\alpha_1\neq\alpha_2\in p\Zp$ and $D$ be a semisimple 4-dimensional $\varphi$-module with $\pchar(\varphi)(X)=(X^2+\alpha_1X+p)(X^2+\alpha_2X+p)$. Then $D=D_1\oplus D_2$ with $D_i=\Ker(\varphi^2+\alpha_i\varphi +p)$, which are the nontrivial $\varphi$-stable subspaces of $D$, and $t_N(D_i)=1$. Since $\alpha_1\neq\alpha_2$ one checks that any $\Qp$-linear $\delta:D^*\rightarrow D$ satisfying $\delta^*=-\delta$ and $\varphi\delta=p\delta\varphi^{*\,-1}$ sends $D_2^{\perp}$ into $D_1$ and $D_1^{\perp}$ into $D_2$. Endowing $D$ with an admissible Hodge-Tate $(0,1)$ filtration such that $\text{\sc (s)}$ does not split amounts to picking a $2$-dimensional subspace $\Fil^1D$ such that $\dim D_1\cap\Fil^1D=1$ and $\dim D_2\cap\Fil^1D=0$ (or vice versa) ; then $\dim D_1\cap\delta(\Fil^1D)^{\perp}=0$ and $\dim D_2\cap\delta(\Fil^1D)^{\perp}=1$, therefore $\delta(\Fil^1D)^{\perp}\neq\Fil^1D$. This shows that the $p$-adic Tate modules of abelian schemes over $\Zp$ with special fibre $\Fp$-isogenous to the product of two nonisogenous supersingular elliptic curves are semisimple. 
\end{rem}

\begin{rem}
One constructs in a similar fashion for each integer $n\geq2$ a lift of the $n$-fold product of a supersingular elliptic curve over $\Fp$ with nonsemisimple $p$-adic Tate module. 
\end{rem}

\section{A lift of a simple supersingular abelian surface}
\label{sec:absurf}
In this section we assume $p\equiv 1\bmod 3\Z$ which is equivalent to $\zeta_3\in\Qp$ where $\zeta_3$ is a primitive 3rd root of unity. Consider the filtered $\varphi$-module $(D,\Fil)$ defined as follows. There is a $\Qp$-basis $\mathcal{B}=(x_1,y_1,x_2,y_2)$ for $D$ so that 
$$D=\Qp x_1\oplus\Qp y_1\oplus\Qp x_2\oplus\Qp y_2$$ 
is a 4-dimensional $\Qp$-vector space. The matrix of $\varphi$ over $\mathcal{B}$ is 
$$\Mat_{\mathcal{B}}(\varphi)=\,\begin{pmatrix}
0 & \zeta_3p & 0 & 0 \\
1 & 0  & 0 & 0 \\
0 & 0  & 0 & \zeta_3^{-1}p \\
0 & 0  & 1 & 0
\end{pmatrix}\in\,GL_4(\Qp)$$
and the filtration is given by 
$$\Fil^0D=D,\quad\Fil^1D=\Qp x_1\oplus\Qp(y_1+x_2),\quad\Fil^2D=0.$$

\begin{prop}
\label{absurfprop}
There is an abelian surface $\mathcal{A}$ over $\Qp$ such that $(D,\Fil)\simeq\D^*_{\text{\em\tiny cris}}(V_p(\mathcal{A}))$. Further 
\begin{itemize}
\item[(a)] $\mathcal{A}$ has good reduction with special fibre isogenous to a supersingular simple abelian surface over $\Fp$
\item[(b)] the $G$-module $V_p(\mathcal{A})$ is not semisimple. 
\end{itemize}
\end{prop}

\begin{proof}
Just as in the proof of proposition~\ref{prodellipprop} we have $t_H(D)=2=t_N(D)$. Since 
$$\pchar(\varphi)(X)=X^4+pX^2+p^2=(X^2-\zeta_3p)(X^2-\zeta_3^{-1}p)$$
the nontrivial sub-$\Qp[\varphi]$-modules of $D$ are the $D_i=\Qp x_i\oplus\Qp y_i$ for $i=1,2$ both having Newton invariant $t_N(D_i)=1$, and Hodge invariants $t_H(D_1)=1$, $t_H(D_2)=0$. Again we obtain a nonsplit exact sequence $\text{\sc (s)}$ in ${\bf MF}_{\Qp}^{\adm}(\varphi)$ and $(D,\Fil)$ is not semisimple.

The action of $\varphi$ is semisimple and $\pchar(\varphi)=\pchar(\text{Fr}_{A})$ where $A$ is a supersingular simple abelian surface over $\Fp$ with $\pchar(\text{Fr}_{A})(X)=X^4+pX^2+p^2$. Thus $(D,\Fil)$ satisfies condition (1) of theorem~\ref{abvarthm} as well as condition (a) of section~\ref{sec:method}. It obviously satisfies (2) and it remains to check (3). Let $\mathcal{B}^*=(x_1^*,y_1^*,x_2^*,y_2^*)$ be the dual basis of $\mathcal{B}$ for $D^*$. Again $(\Fil^1D)^{\perp}=\Qp y_2^*\oplus\Qp(y_1^*-x_2^*)$ and the matrix of $p\varphi^{*\,-1}$ over $\mathcal{B}^*$ is 
 $$p\Mat_{\mathcal{B}}(\varphi^{-1})^{t}=\,\begin{pmatrix}
0 & \zeta_3^{-1} & 0 & 0 \\
p & 0  & 0 & 0 \\
0 & 0  & 0 & \zeta_3 \\
0 & 0  & p & 0
\end{pmatrix}.$$
Let $\delta:D^*\rightarrow D$ be the $\Qp$-linear morphism with matrix over the bases $\mathcal{B}^*$ and $\mathcal{B}$  
$$\Mat_{\mathcal{B}^*\!\!,\mathcal{B}}(\delta)=\,\begin{pmatrix}
0 & 0 & 0 & \zeta_3 \\
0 & 0  & 1 & 0 \\
0 & -1  & 0 & 0 \\
-\zeta_3 & 0  & 0 & 0
\end{pmatrix}.$$
As in the proof of proposition~\ref{prodellipprop} one checks that $\delta$ is invertible, satisfies $\delta^*=-\delta$, $\varphi\delta=p\delta\varphi^{*\,-1}$, and that $\delta(\Fil^1D)^{\perp}=\Fil^1D$. 
\end{proof}

\begin{rem}
The objects $(D_1,\Fil^iD\cap D_1)$ and $(D_2,\Fil^iD\bmod D_1)$ in ${\bf MF}_{\Qp}^{\adm}(\varphi)$ do not arise from elliptic schemes over $\Zp$, however~\cite{Ki} Thm.0.3 shows the existence of $p$-divisible groups $\mathcal{G}_i$ over $\Zp$ such that $D_i\simeq\Dcr(V_p(\mathcal{G}_i))$. The special fibre of $\mathcal{A}(p)$ is $\Fp$-isogenous to the product of the special fibres of the $\mathcal{G}_i$, themselves being nonisogenous. Thus we obtain a nonsplit exact sequence of $G$-modules 
$$\xymatrix{1\ar[r]  &  V_p(\mathcal{G}_2)\ar[r] & V_p(\mathcal{A})\ar[r] &  V_p(\mathcal{G}_1)\ar[r] & 1}$$
and Tate's full faithfulness theorem shows that $\mathcal{A}(p)$ is not $\Zp$-isogenous to $\mathcal{G}_1\times\mathcal{G}_2$.
\end{rem}

\begin{rem}
Starting with $X^4-pX^2+p^2$ when $p\equiv 1\bmod 3\Z$ and $X^4+p^2$ when $p\equiv 1\bmod 4\Z$ one obtains alike nonsemisimple $4$-dimensional supersingular representations (just replace $\zeta_3$ by $\zeta_6$ or $\zeta_4$). More generally the 
$$p^d\Phi_n\Bigl(\frac{X^2}{p}\Bigr)=\prod_{i\in(\Z/n\Z)^{\times}}(X^2-\zeta_n^ip)\qquad\text{ with }\,d=\#\bigl(\Z/n\Z\bigr)^{\times}\geq 2$$
where $\Phi_n$ is the $n$th cyclotomic polynomial are supersingular $p$-Weil polynomials leading when $p\equiv 1\bmod n\Z$ to similar higher-dimensional constructions. 
\end{rem}

\bigskip

\end{document}